\documentclass[11pt]{article}
\usepackage[T1]{fontenc}
\usepackage[utf8]{inputenc}
\usepackage{lmodern}
\usepackage[margin=1in]{geometry}
\usepackage{amsmath,amssymb,amsthm}
\usepackage{microtype}
\usepackage{tikz}
\usetikzlibrary{arrows.meta,positioning}
\usepackage[hidelinks]{hyperref}

\newtheorem{theorem}{Theorem}[section]
\newtheorem{lemma}[theorem]{Lemma}
\newtheorem{proposition}[theorem]{Proposition}
\newtheorem{corollary}[theorem]{Corollary}
\theoremstyle{definition}
\newtheorem{definition}[theorem]{Definition}
\theoremstyle{remark}

\numberwithin{equation}{section}
\allowdisplaybreaks[1]
\hypersetup{
pdftitle={A Hierarchical Online-Compression Approach to Proportional Subsets of Bounded Orthonormal Systems},
pdfauthor={Will Burstein, Alex Iosevich, Ben Krause, Siddharth Mulherkar}
}

\title{A Hierarchical Online-Compression Approach to\\
Proportional Subsets of Bounded Orthonormal Systems}
\author{Will Burstein \and Alex Iosevich\thanks{Alex Iosevich acknowledges partial support from the National Science
Foundation under grant DMS-2506858. Ben Krause acknowledges partial
support from the Engineering and Physical Sciences Research Council
through New Investigator Award EP/W010275/2 and from UK Research and
Innovation through Horizon Europe Guarantee grant EP/Y007336/1 for
his ERC Starting Grant project PCMEA (Pointwise Convergence of Multiple
Ergodic Averages).} \and Ben Krause}
\date{}

\begin{document}
\maketitle

\begin{abstract}
Let $(\phi_i)_{i=1}^n$ be an orthonormal system on a probability space with $\|\phi_i\|_\infty\leq L$. We prove that, for every $2\leq k<n$, there is a set $I$ of at least $n-k$ indices such that
$$
\|f\|_2\leq CM_k\sqrt{\log(e+M_k)}\|f\|_1
$$
for every $f$ in the span of $(\phi_i)_{i\in I}$, where $M_k=L\sqrt{\frac{n}{k}\log(2k)}$. This replaces the logarithmic exponent $\frac{5}{2}$ in the corresponding theorem of Gu\'edon, Mendelson, Pajor, and Tomczak-Jaegermann by $\frac{1}{2}$. We also obtain a partition into two sets of cardinality $\frac{n}{2}+O(\sqrt n)$ for which both spans satisfy the resulting norm comparison, and for uniformly bounded systems the proportional estimate is $C\sqrt{\log n\log\log n}$ for sufficiently large $n$.

This is the third paper in a series systematically using gradient descent as a proof mechanism in harmonic analysis, following the $\Lambda(p)$ and Bernoulli restricted isometry papers of Burstein, Iosevich, and Krause, and the descent here is entropic mirror descent. We represent each function in an $L_1$ ball by a probability measure of controlled relative entropy, and this entropy pays for the corrections to the predictions of its coefficients. Processing the thresholds from large to small, we obtain a single weighted bound for all the corrections, which we turn into a finite description of approximate coefficient squares and then into a conditional Rademacher estimate, and symmetrization completes the argument directly on the $L_1$ ball. We also record consequences for concentration and for recovery from coefficient measurements.
\end{abstract}

\section{Introduction}

Orthogonality tells us exactly how to calculate the $L_2$ norm of a linear combination, since for an orthonormal system $(\phi_i)_{i=1}^n$,
$$
\left\|\sum_{i=1}^n a_i\phi_i\right\|_2^2
=\sum_{i=1}^n|a_i|^2.
$$
There is no comparable identity for the $L_1$ norm. A function can have a substantial $L_2$ norm because it is large on a small set, while its $L_1$ norm remains small. Can we remove a controlled number of functions from the system so that this concentration becomes less severe for every linear combination of the remaining functions?

The words ``every linear combination'' are essential. Choosing a favorable subsystem for a single function is a much easier problem, whereas here the subsystem must be chosen once, and the resulting inequality must hold for all coefficient vectors. This is the same uniformity which appears in Bourgain's $\Lambda(p)$ selection theorem \cite{Bourgain1989} and in the restricted isometry problem \cite{Bourgain2014}.

Throughout the paper, $(\Omega,\mu)$ is a probability space and $(\phi_i)_{i=1}^n$ is a real or complex orthonormal system in $L_2(\mu)$ satisfying
\begin{equation}\label{eq:bounded-system}
\|\phi_i\|_\infty\leq L
\end{equation}
for $1\leq i\leq n$, so that in particular, $L\geq1$. We write $[n]=\{1,\ldots,n\}$,
$$
E=\operatorname{span}\{\phi_1,\ldots,\phi_n\},
$$
and $E_I=\operatorname{span}\{\phi_i:i\in I\}$, and all spans may be taken over $\mathbb C$, even when the original functions are real. For $1\leq p<\infty$ we use
$$
\|f\|_p=\left(\int_\Omega |f|^pd\mu\right)^{\frac{1}{p}},
$$
while $\|f\|_\infty$ denotes the essential supremum. For a set $A$, the function $1_A$ is its indicator, and all logarithms are natural. The letters $c$ and $C$ denote positive absolute constants whose values may change from one occurrence to another, and a subscript indicates a permitted dependence.

Gu\'edon, Mendelson, Pajor, and Tomczak-Jaegermann developed this selection problem in \cite{GMPT2007,GMPT2008}. Their Theorem 3 in \cite{GMPT2008} provides a set $I$ with $|I|\geq n-k$ and an $L_2$ to $L_1$ comparison with factor
\begin{equation}\label{eq:old-factor}
C\mu_k(\log\mu_k)^{\frac{5}{2}},
\end{equation}
where $\mu_k=L\sqrt{\frac{n}{k}}\sqrt{\log k}$. Our result improves the power of the second logarithm in this estimate, and we use $\log(2k)$ and $\log(e+M_k)$ in order to state the bound without exceptions for small parameters.

\begin{theorem}\label{thm:main}
There is an absolute constant $C$ with the following property. Let $(\phi_i)_{i=1}^n$ be an orthonormal system satisfying \eqref{eq:bounded-system}, suppose that $2\leq k<n$, and define
\begin{equation}\label{eq:Mk}
M_k=L\sqrt{\frac{n}{k}\log(2k)}.
\end{equation}
Then there is a set $I\subset[n]$ with $|I|\geq n-k$ such that
\begin{equation}\label{eq:main-bound}
\|f\|_2\leq CM_k\sqrt{\log(e+M_k)}\|f\|_1
\end{equation}
for every $f\in E_I$. If $n\geq2$ is even, there is a set $I\subset[n]$ with
\begin{equation}\label{eq:half-size}
\left||I|-\frac{n}{2}\right|\leq C\sqrt n
\end{equation}
such that both $E_I$ and $E_{[n]\setminus I}$ satisfy
\begin{equation}\label{eq:half-bound}
\|f\|_2\leq
CL\sqrt{\log(2n)\log(e+L\sqrt{\log(2n)})}\|f\|_1.
\end{equation}
\end{theorem}

The reverse inequality $\|f\|_1\leq\|f\|_2$ follows from the Cauchy--Schwarz inequality because $\mu(\Omega)=1$, so the theorem provides an equivalence of the two norms on each selected span. If $L$ is bounded by an absolute constant and $\alpha n\leq k<n$ for a fixed $0<\alpha<1$, the factor in \eqref{eq:main-bound} is bounded by
$$
C_\alpha\sqrt{\log n\log\log n}
$$
for sufficiently large $n$, whereas the corresponding bound from \eqref{eq:old-factor} is $C_\alpha\sqrt{\log n}(\log\log n)^{\frac{5}{2}}$. The improvement concerns this specific comparison, and we do not claim that the remaining factor $\sqrt{\log\log n}$ is necessary.

These norm comparisons also have a direct interpretation in signal recovery. Iosevich and Mayeli \cite{IM2025} use restriction estimates and Bourgain's $\Lambda_q$ theorem to obtain uncertainty principles and recovery from incomplete Fourier data, and Iosevich, Kashin, Limonova, and Mayeli \cite{IKLM2024} use large subsystems of bounded orthonormal bases to obtain uniform sparse recovery by $\ell^1$ minimization. In our setting $E_I$ is the kernel of the coefficient measurements on $[n]\setminus I$. The improved norm comparison limits how much $L_1$ mass a vector in this kernel can place on a small support, which leads to exact recovery of sparse signals, and Section~\ref{sec:recovery} makes this implication explicit.

\subsection{The third paper in the descent program}

This is the third paper in the program of Burstein, Iosevich, and Krause in which descent is used to prove uniform estimates in harmonic analysis. The first paper \cite{BIK1} contains a greedy proof of Bourgain's $\Lambda(p)$ selection theorem with prescribed polynomial failure bounds. For $N\geq2$ pairwise orthogonal functions bounded by one, it retains $\lceil N^{\frac{2}{p}}\rceil$ functions with success probability at least $1-N^{-A}$, for any fixed $p>2$ and $A>0$, the norm constant depending on $p$ and $A$ and the same subsystem working for all complex coefficient vectors, and its proof uses a decreasing convex potential to control the corrections at all accuracy levels together.

The second paper \cite{BIK2} uses entropic mirror descent to prove improved restricted isometry estimates for sampled Fourier and Hadamard matrices. It controls the squared norm of every sparse vector on a common event, with polynomial failure bounds, and obtains recovery results from the resulting restricted isometry property. The present paper treats the comparison of the $L_1$ and $L_2$ norms on large coordinate spans of a bounded orthonormal system.

The companion papers supply the common method and its earlier applications, but we prove all the estimates needed here, including the descent and probability estimates, and in particular the main theorem below does not assume a selection or restricted isometry theorem from either companion paper.

There is a common idea behind these problems. If an approximation fails a test, that failure should identify a direction in which a nonnegative potential decreases, and a lower bound for the decrease limits how often failure can occur. When each correction has a finite description, this deterministic bound becomes a bound for the number of descriptions which must be considered in a probabilistic argument.

Here the approximation is a probability measure $P$, and the potential is the relative entropy $D(Q_f\|P)$, defined in Section~\ref{sec:representation}, where $Q_f$ represents the target function. The update is the exponential-weights form of entropic mirror descent, and we prove directly that a mistaken prediction at threshold $t$ decreases the potential by at least
$$
c\frac{\eta^2t^2}{L^2},
$$
where $\eta>0$ is a fixed accuracy parameter. Summing these decreases over all the thresholds, we obtain the weighted correction bound which drives the proof. We also verify the variational identity defining mirror descent, so that the role of gradient descent is expressed by an exact formula rather than only by an analogy. General background on entropic mirror descent may be found in the paper of Beck and Teboulle \cite{BeckTeboulle}, and the estimates needed here are proved in full.

The point specific to the present paper is that we work directly with the $L_1$ ball. A function with $\|f\|_1\leq1$ and $\|f\|_2=\rho$ admits a representation whose initial entropy is at most $C\log(e+\rho)$, and this representation replaces the approximation of $L_1$ by $L_p$, $p>1$, in the proof of \cite{GMPT2008}.

\subsection{How the proof is organized}

Let
$$
a_i(f)=\int_\Omega f\overline{\phi_i}d\mu.
$$
We sample $k$ indices $J_1,\ldots,J_k$ independently and uniformly from $[n]$, and for a fixed function $f \in E$, Parseval's identity shows that
$$
\mathbb E\sum_{\ell=1}^k|a_{J_\ell}(f)|^2
=\frac{k}{n}\|f\|_2^2.
$$
The main task is to make this estimate uniform over a sphere inside the $L_1$ ball. Once this is done, a function supported on the unsampled coefficients cannot have a large $L_2$ norm and a small $L_1$ norm, since its sampled energy is zero, whereas the uniform estimate says that this energy must be positive.

To obtain uniformity, we predict whether each real and imaginary coefficient lies above, below, or near a threshold, predictions close to the boundary being allowed to remain undecided. A stable incorrect prediction triggers a descent step, and we retain the same predictor while moving from larger thresholds to smaller ones, so that only the positions and labels of the corrections need to be recorded.

Why do we process the thresholds in this order? The largest threshold at which a coefficient receives a nonzero label determines an approximation to its square, and consequently the band at a given threshold can be reconstructed from the histories at that threshold and at all larger thresholds, the lower thresholds being irrelevant to that band. This observation allows the cost of each history to be summed with geometric weights. After conditioning on the sampled indices, a union bound over these histories controls a Rademacher process, and a final symmetrization argument produces the uniform empirical estimate. Figure~\ref{fig:outline} shows how these steps depend on one another.

\begin{figure}[htbp]
\centering
\begin{tikzpicture}[font=\small,>={Stealth},node distance=0.38cm,
box/.style={draw,rounded corners,align=center,minimum width=12.6cm,inner sep=4pt}]
\node[box] (a) {The representing measure $Q_f$ of Lemma~\ref{lem:comparator}, with $D(Q_f\|\nu)\leq C\log(e+\rho)$};
\node[box,below=of a] (b) {The descent bound of Proposition~\ref{prop:mistakes}, $\sum_{\sigma,j}t_j^2d_j^{(\sigma)}(f)\leq C_\eta L^2\log(e+\rho)$};
\node[box,below=of b] (c) {Decoded squares and bands, Lemmas~\ref{lem:reconstruction}--\ref{lem:tail},\\each band depending only on the histories at its own and higher levels};
\node[box,below=of c] (d) {The weighted cost of Proposition~\ref{prop:weighted-cost}, $\sum_{\sigma,j}t_j^2q_j^{(\sigma)}(f)\leq K_\rho$};
\node[box,below=of d] (e) {The conditional Rademacher bound $C\sqrt{K_\rho V_\rho}$ of Proposition~\ref{prop:rademacher}};
\node[box,below=of e] (f) {Symmetrization and the estimate $\mathcal D_\rho\leq\epsilon\delta\rho^2$ of Proposition~\ref{prop:empirical}};
\node[box,below=of f] (g) {The choice of the radius and Lemma~\ref{lem:radius}, which prove Theorem~\ref{thm:main}};
\draw[->] (a)--(b);
\draw[->] (b)--(c);
\draw[->] (c)--(d);
\draw[->] (d)--(e);
\draw[->] (e)--(f);
\draw[->] (f)--(g);
\end{tikzpicture}
\caption{The structure of the proof. The entropy of the representing measure pays for the corrections, the histories of the corrections determine approximate coefficient squares, and their weighted cost controls the Rademacher process to which symmetrization reduces the empirical second moment.}
\label{fig:outline}
\end{figure}
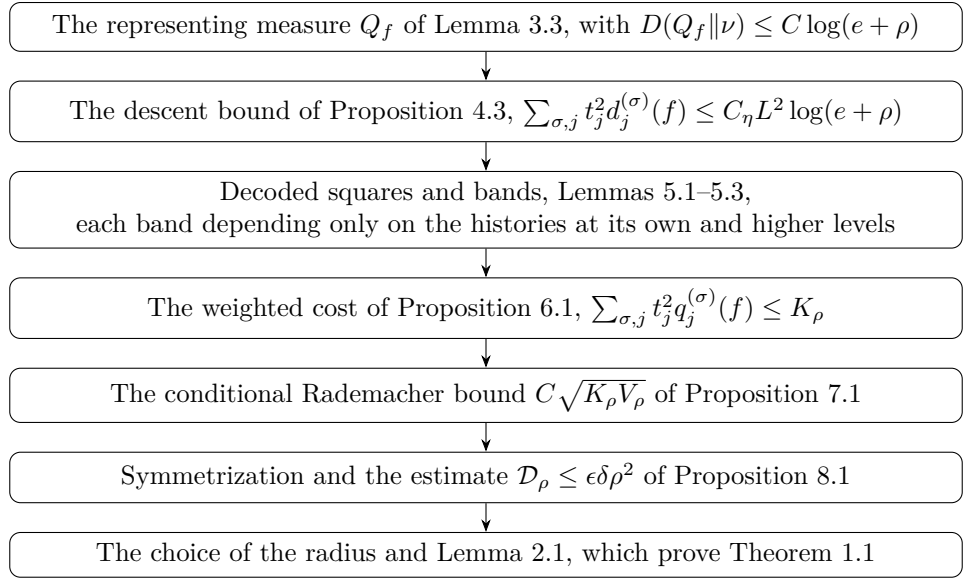

The construction yields a successful sample with probability at least $\frac{97}{100}$ for the radius used below. Repeated samples are permitted, and the retained subsystem is the complement of the set of distinct sampled indices. The proof provides an existence theorem and a probability statement for this sampling model, but it does not provide an efficient deterministic procedure for finding the subsystem, or a probability estimate for a uniformly chosen set of prescribed cardinality.

\section{From sampled coefficients to a norm comparison}

We first isolate the elementary geometric step at the end of the argument. For every $f\in E$, orthonormality and \eqref{eq:bounded-system} show that
\begin{equation}\label{eq:coefficient-bounds}
\sum_{i=1}^n|a_i(f)|^2=\|f\|_2^2
\text{ and }
|a_i(f)|\leq L\|f\|_1,
\end{equation}
and in particular
\begin{equation}\label{eq:trivial-bound}
\|f\|_2\leq L\sqrt n\|f\|_1.
\end{equation}
This handles the range of parameters in which the asserted bound is larger than the elementary one.

For independent uniform indices $J_1,\ldots,J_k$ in $[n]$, we put
\begin{equation}\label{eq:sampling-map}
\Gamma f=(a_{J_1}(f),\ldots,a_{J_k}(f))
\end{equation}
and $\delta=\frac{k}{n}$, and we write $J=(J_1,\ldots,J_k)$ for the whole sampled sequence. The norm on the range of $\Gamma$ is the usual unnormalized Euclidean norm, so that
\begin{equation}\label{eq:isotropy}
\mathbb E\|\Gamma f\|_2^2=\delta\|f\|_2^2.
\end{equation}
For $\rho>0$ we define
\begin{equation}\label{eq:sphere}
\mathcal F_\rho=\{f\in E:\|f\|_1\leq1, \|f\|_2=\rho\}.
\end{equation}
If this set is empty, every function in the $L_1$ unit ball of $E$ has $L_2$ norm strictly less than $\rho$, since if $\|f\|_1\leq1$ and $\|f\|_2\geq\rho$, then $\frac{\rho}{\|f\|_2}f$ would belong to $\mathcal F_\rho$. In that case the desired norm comparison already holds on all of $E$, and we therefore assume that $\mathcal F_\rho$ is nonempty whenever we take a supremum over it.

\begin{lemma}\label{lem:radius}
Suppose that $0<c_0<1$ and that a realization of the sample satisfies
\begin{equation}\label{eq:sphere-event}
\sup_{f\in\mathcal F_\rho}
|\|\Gamma f\|_2^2-\delta\rho^2|
\leq c_0\delta\rho^2,
\end{equation}
and let $S$ be the set of distinct sampled indices. Then $\|f\|_2\leq\rho\|f\|_1$ for every $f\in E_{[n]\setminus S}$, and if also $(1+c_0)\delta<1$, the same inequality holds for every $f\in E_S$.
\end{lemma}

\begin{proof}
Since $(\phi_i)_{i=1}^n$ is an orthonormal basis of $E$, every $f\in E$ has the unique expansion $f=\sum_ia_i(f)\phi_i$, so $\Gamma f=0$ precisely when $a_i(f)=0$ for all $i\in S$, and the kernel of $\Gamma$ is $E_{[n]\setminus S}$. Suppose that a nonzero $f$ in this kernel satisfies $\|f\|_2\geq\rho\|f\|_1$. Then
$$
g=\frac{\rho}{\|f\|_2}f
$$
has $\|g\|_2=\rho$ and $\|g\|_1=\frac{\rho\|f\|_1}{\|f\|_2}\leq1$, so it belongs to $\mathcal F_\rho$ and still lies in the kernel, and the expression in \eqref{eq:sphere-event} at $g$ equals $\delta\rho^2$, contradicting $c_0<1$.

Now suppose instead that a nonzero $f\in E_S$ satisfies $\|f\|_2\geq\rho\|f\|_1$, and define $g$ in the same way. Every index in $S$ occurs at least once in the sample, so that
$$
\|\Gamma g\|_2^2
=\sum_{\ell=1}^k|a_{J_\ell}(g)|^2
\geq\sum_{i\in S}|a_i(g)|^2
=\rho^2,
$$
while \eqref{eq:sphere-event} shows that
$$
\|\Gamma g\|_2^2\leq(1+c_0)\delta\rho^2<\rho^2,
$$
and this is again a contradiction. Repeated indices cause no difficulty, since they only increase the sampled energy in the last argument.
\end{proof}

\section{Representing a function by a probability measure}\label{sec:representation}

The descent argument will predict coefficients by averaging bounded functions against a probability measure, so we need to represent the target coefficients in the same way, with a controlled cost for the representing measure. The relevant cost is relative entropy.

\begin{definition}
If $Q$ and $P$ are probability measures and $Q$ is absolutely continuous with respect to $P$, their relative entropy is
$$
D(Q\|P)=\int\log\left(\frac{dQ}{dP}\right)dQ,
$$
where we use the convention $0\log0=0$, and we set $D(Q\|P)=+\infty$ when absolute continuity fails. Relative entropy is nonnegative, as one sees by applying Jensen's inequality to $h\log h$, where $h=\frac{dQ}{dP}$, together with the identity $\int hdP=1$.
\end{definition}

Complex coefficients require a small preliminary observation. Let $\mathcal U=\{1,-1,i,-i\}$. The square with vertices $\sqrt2\mathcal U$ contains the unit circle, and we shall use an explicit form of this fact.

\begin{lemma}\label{lem:phases}
For each complex number $\omega$ with $|\omega|=1$, there are measurably chosen nonnegative numbers $(\lambda_\zeta(\omega))_{\zeta\in\mathcal U}$ such that
\begin{equation}\label{eq:phase-identity}
\sum_{\zeta\in\mathcal U}\lambda_\zeta(\omega)=1
\text{ and }
\sqrt2\sum_{\zeta\in\mathcal U}\lambda_\zeta(\omega)\zeta=\omega.
\end{equation}
\end{lemma}

\begin{proof}
Write $\omega=x+iy$, let $x_+=\max(x,0)$ and $x_-=\max(-x,0)$, and define $y_+$ and $y_-$ in the same way. Since $x^2+y^2=1$, the Cauchy--Schwarz inequality shows that $|x|+|y|\leq\sqrt2$, and therefore
$$
r=1-\frac{|x|+|y|}{\sqrt2}
$$
is nonnegative. We set $\lambda_1=\frac{x_+}{\sqrt2}+\frac{r}{4}$, $\lambda_{-1}=\frac{x_-}{\sqrt2}+\frac{r}{4}$, $\lambda_i=\frac{y_+}{\sqrt2}+\frac{r}{4}$, and $\lambda_{-i}=\frac{y_-}{\sqrt2}+\frac{r}{4}$. These numbers sum to one, and subtracting opposite weights we obtain the real and imaginary parts of the second identity in \eqref{eq:phase-identity}. The formulas are continuous in $x$ and $y$, so measurability also follows.
\end{proof}

Let $\star$ be an extra point, distinct from the points of $\Omega\times\mathcal U$, and form the disjoint union
$$
\mathcal E=\{\star\}\sqcup(\Omega\times\mathcal U),
$$
which we equip with the probability measure
\begin{equation}\label{eq:prior}
\nu=\frac{1}{2}\delta_\star+
\frac{1}{8}\sum_{\zeta\in\mathcal U}(\mu\otimes\delta_\zeta),
\end{equation}
where $\delta_e$ denotes the unit point mass at $e$. For each coordinate $i$ we define $U_i(\star)=0$ and
\begin{equation}\label{eq:experts}
U_i(t,\zeta)=\sqrt2\zeta\overline{\phi_i(t)}.
\end{equation}
We call a point of $\mathcal E$ an expert because it supplies a prediction $U_i(e)$ for each coordinate, a terminology which serves only to describe averages of the functions $U_i$ and introduces no additional assumption. We choose bounded representatives of the finitely many functions $\phi_i$, after changing them on a common null set, and consequently
\begin{equation}\label{eq:expert-bound}
|U_i(e)|\leq B
\end{equation}
at every point used below, where $B=\sqrt2L$.

\begin{lemma}\label{lem:comparator}
For every $f\in\mathcal F_\rho$ there is a probability measure $Q_f$ on $\mathcal E$ such that
\begin{equation}\label{eq:comparator-mean}
\int U_idQ_f=a_i(f)
\end{equation}
for $1\leq i\leq n$ and
\begin{equation}\label{eq:entropy-bound}
D(Q_f\|\nu)\leq C\log(e+\rho).
\end{equation}
\end{lemma}

\begin{proof}
Put $s=\|f\|_1$, so that $0<s\leq1$, since $\rho>0$, and write $f(t)=|f(t)|\omega(t)$, choosing $\omega(t)=1$ where $f(t)=0$. We assign the mass $1-s$ to $\star$, and give $Q_f$ the density $|f(t)|\lambda_\zeta(\omega(t))$ with respect to $\mu$ on the copy of $\Omega$ indexed by $\zeta$. Its total mass is
$$
(1-s)+\int_\Omega|f|\sum_\zeta\lambda_\zeta d\mu=1,
$$
and it is absolutely continuous with respect to $\nu$, with density $2(1-s)$ at $\star$ and $8|f(t)|\lambda_\zeta(\omega(t))$ on each copy of $\Omega$. The second identity in \eqref{eq:phase-identity} shows that
$$
\int U_idQ_f
=\int_\Omega |f(t)|\omega(t)\overline{\phi_i(t)}d\mu(t)
=a_i(f).
$$

It remains to estimate the entropy. Directly from \eqref{eq:prior}, $D(Q_f\|\nu)$ is the sum of
$$
(1-s)\log(2(1-s))+s\log8+\int_\Omega |f|\log|f|d\mu
$$
and the integral
$$
\int_\Omega |f|\sum_{\zeta\in\mathcal U}
\lambda_\zeta(\omega)\log\lambda_\zeta(\omega)d\mu.
$$
Each $\lambda_\zeta$ lies in $[0,1]$, so the last integral is nonpositive, and it is finite because each term $\lambda_\zeta\log\lambda_\zeta$ lies in $[-\frac{1}{e},0]$, while the first two terms of the sum are bounded above by an absolute constant. Thus
\begin{equation}\label{eq:entropy-intermediate}
D(Q_f\|\nu)\leq C+\int_\Omega |f|\log|f|d\mu,
\end{equation}
these integrals being finite because $t\log t\geq-\frac{1}{e}$ for $t\geq0$ and its positive part is bounded by $t^2$.

Let $q=\frac{|f|}{s}$. The measure $qd\mu$ is a probability measure which assigns zero mass to the set where $q=0$, and the integral of $|\log q|$ against it is finite, by the same bounds for $q\log q$. Jensen's inequality for the concave function $\log$, with respect to this measure, shows that
$$
\int_\Omega q\log qd\mu
\leq\log\left(\int_\Omega q^2d\mu\right)
=\log\left(\frac{\rho^2}{s^2}\right),
$$
and consequently
$$
\int_\Omega |f|\log|f|d\mu
=s\log s+s\int_\Omega q\log qd\mu
\leq s\log\left(\frac{\rho^2}{s}\right).
$$
Finally,
$$
s\log\left(\frac{\rho^2}{s}\right)
\leq2\max(\log\rho,0)+\frac{1}{e}
\leq C\log(e+\rho),
$$
and together with \eqref{eq:entropy-intermediate} this proves \eqref{eq:entropy-bound}.
\end{proof}

The logarithm in \eqref{eq:entropy-bound} has a concrete interpretation. A function of $L_1$ norm at most one can be concentrated, but its prescribed $L_2$ norm controls the entropy needed to describe that concentration, and the proof below uses this amount of entropy only once, across all the thresholds.

\section{Entropic descent and the cost of a correction}

We begin with the elementary exponential estimate behind the descent calculation, and we include its proof so that the constants and the direction of the entropy are explicit.

\begin{lemma}\label{lem:exponential}
Let $X$ be a real random variable with $|X|\leq M$. For every real $v$,
$$
\log\mathbb E e^{vX}\leq v\mathbb E X+\frac{v^2M^2}{2}.
$$
\end{lemma}

\begin{proof}
Write $F(v)=\log\mathbb E e^{vX}$. Boundedness permits differentiation under the expectation, and
$$
F''(v)=\frac{\mathbb E(X^2e^{vX})}{\mathbb E e^{vX}}
-\left(\frac{\mathbb E(Xe^{vX})}{\mathbb E e^{vX}}\right)^2
$$
is the variance of $X$ under the probability measure with density $\frac{e^{vX}}{\mathbb E e^{vX}}$, which is at most the second moment, and hence at most $M^2$. Since $F(0)=0$ and $F'(0)=\mathbb E X$, Taylor's formula in the form
$$
F(v)=v\mathbb E X+v^2\int_0^1(1-t)F''(tv)dt
$$
proves the assertion for both signs of $v$.
\end{proof}

Suppose that $P$ is a current probability measure, that $h$ is a bounded real function on $\mathcal E$, and that $\lambda>0$. We define its update $P^+$ by
\begin{equation}\label{eq:update}
\frac{dP^+}{dP}(e)=\frac{e^{-\lambda h(e)}}{Z},
\end{equation}
where $Z=\int e^{-\lambda h}dP$. This is the exponential-weights update, and it is also precisely an entropic mirror descent step, since for any probability measure $R$ of finite relative entropy with respect to $P$,
\begin{equation}\label{eq:variational}
\lambda\int hdR+D(R\|P)
=D(R\|P^+)-\log Z.
\end{equation}
To check this identity, we use $\log\frac{dP^+}{dP}=-\lambda h-\log Z$ when expanding $D(R\|P^+)$. The right side is minimized when $R=P^+$, since relative entropy is nonnegative and vanishes there, so the new measure minimizes a linear loss plus an entropy cost for moving away from the previous measure, which is the counterpart of minimizing a linear loss plus a squared Euclidean distance in a projected gradient step.

\begin{lemma}\label{lem:one-step}
Suppose that $|h|\leq M$ and $D(Q\|P)<\infty$. Then the update \eqref{eq:update} satisfies
\begin{equation}\label{eq:entropy-drop}
D(Q\|P^+)-D(Q\|P)
\leq-\lambda\left(\int hdP-\int hdQ\right)
+\frac{\lambda^2M^2}{2}.
\end{equation}
\end{lemma}

\begin{proof}
Because $|h|\leq M$, both $Z$ and $e^{-\lambda h}$ lie between $e^{-\lambda M}$ and $e^{\lambda M}$, and therefore
$$
e^{-2\lambda M}\leq\frac{dP^+}{dP}\leq e^{2\lambda M}.
$$
In particular, $P^+$ and $P$ have the same null sets and $D(Q\|P^+)$ is finite. Taking logarithms in \eqref{eq:update} and integrating against $Q$, we obtain the exact identity
$$
D(Q\|P^+)-D(Q\|P)=\lambda\int hdQ+\log Z,
$$
and Lemma~\ref{lem:exponential}, applied to $-h$ under $P$, bounds $\log Z$ by $-\lambda\int hdP+\frac{\lambda^2M^2}{2}$, which proves \eqref{eq:entropy-drop} after substitution.
\end{proof}

\subsection{Thresholds and stable labels}

Fix $0<\eta<10^{-3}$, whose final value we choose only after specifying the accuracy needed in the empirical estimate. Suppose that $0<\rho\leq L\sqrt n$, and set
\begin{equation}\label{eq:thresholds}
\tau^2=\eta\frac{\rho^2}{n}
\end{equation}
and $t_j=\tau(1+\eta)^j$ for $0\leq j\leq H$, where $H$ is the smallest nonnegative integer such that
\begin{equation}\label{eq:top-threshold}
t_H\geq\frac{L}{1-4\eta}.
\end{equation}
Since $\tau\leq\sqrt\eta L<\frac{L}{1-4\eta}$, we have $H\geq1$, and minimality shows that
$$
\frac{L}{1-4\eta}\leq t_H<\frac{(1+\eta)L}{1-4\eta}.
$$
Moreover, for $0\leq r\leq H$,
$$
\sum_{j=0}^r t_j^2
=t_r^2\sum_{m=0}^r(1+\eta)^{-2m}
\leq\frac{t_r^2}{1-(1+\eta)^{-2}},
$$
and consequently
\begin{equation}\label{eq:geometric-sums}
T=\sum_{j=0}^H t_j^2\leq C_\eta L^2
\text{ and }
\sum_{j=0}^r t_j^2\leq C_\eta t_r^2
\end{equation}
for $0\leq r\leq H$. The constants in these bounds do not depend on the number $H$ of thresholds. The choice of the bottom threshold also ensures that $k\tau^2=\eta\delta\rho^2$, so an error of size $C\tau^2$ at each sampled position adds up to a small multiple of the expected energy at radius $\rho$.

Fix $f\in\mathcal F_\rho$ and the sampled indices for the rest of this section. We run one predictor for the real parts of the sampled coefficients and another for their imaginary parts, as separate deterministic procedures, and we make no independence assumption about the two parts of a coefficient. For $\sigma=1,2$ we put $z_\ell^{(1)}=\operatorname{Re}a_{J_\ell}(f)$ and $z_\ell^{(2)}=\operatorname{Im}a_{J_\ell}(f)$, and the functions predicted by the two procedures are respectively $u_\ell^{(1)}(e)=\operatorname{Re}U_{J_\ell}(e)$ and $u_\ell^{(2)}(e)=\operatorname{Im}U_{J_\ell}(e)$. In either case $|u_\ell^{(\sigma)}|\leq B$ and $\int u_\ell^{(\sigma)}dQ_f=z_\ell^{(\sigma)}$, and we describe one procedure, suppressing the superscript $\sigma$ for the moment.

At position $\ell$ and threshold $t_j$, we assign the target label $c_f(\ell,j)=+$ if $z_\ell\geq(1+4\eta)t_j$, the label $-$ if $z_\ell\leq-(1+4\eta)t_j$, the label $0$ if $|z_\ell|\leq(1-4\eta)t_j$, and the label $*$ in the remaining case, that is, when
\begin{equation}\label{eq:labels}
(1-4\eta)t_j<|z_\ell|<(1+4\eta)t_j.
\end{equation}
The label $*$ marks a boundary region, shown in Figure~\ref{fig:labels}, where we impose no requirement, and the other three labels are called stable. This gap around the threshold is useful because a stable mistake will force a definite change in the predictor, while boundary values can be left unresolved without substantially changing the coefficient square, as Lemma~\ref{lem:quantization} will show.

\begin{figure}[htbp]
\centering
\begin{tikzpicture}[x=1cm,y=1cm,font=\small,>={Stealth}]
 \fill[gray!25] (-5,-0.2) rectangle (-3,0.2);
 \fill[gray!25] (3,-0.2) rectangle (5,0.2);
 \draw[->] (-6.8,0)--(6.8,0) node[right] {$z_\ell$};
 \foreach \x in {-5,-3,3,5} {\draw (\x,0.25)--(\x,-0.25);}
 \draw (0,0.15)--(0,-0.15);
 \node[below] at (0,-0.25) {$0$};
 \node[below] at (3,-0.25) {$(1-4\eta)t_j$};
 \node[below] at (-3,-0.25) {$-(1-4\eta)t_j$};
 \draw[gray] (5,-0.25)--(5,-0.85);
 \draw[gray] (-5,-0.25)--(-5,-0.85);
 \node[below] at (5,-0.85) {$(1+4\eta)t_j$};
 \node[below] at (-5,-0.85) {$-(1+4\eta)t_j$};
 \draw[dashed] (4,0.6)--(4,-0.2);
 \draw[dashed] (-4,0.6)--(-4,-0.2);
 \node[above] at (4,0.6) {$t_j$};
 \node[above] at (-4,0.6) {$-t_j$};
 \node at (3.5,0.45) {$*$};
 \node at (-3.5,0.45) {$*$};
 \node[above] at (-6,0.25) {label $-$};
 \node[above] at (0,0.25) {label $0$};
 \node[above] at (6,0.25) {label $+$};
\end{tikzpicture}
\caption{The four labels at one threshold $t_j$ for a real coefficient $z_\ell$, with $\eta$ exaggerated. No requirement is imposed in the shaded boundary regions, which carry the label $*$, and outside them the stable label leaves a margin of $4\eta t_j$ between the correct score and each incorrect one.}
\label{fig:labels}
\end{figure}
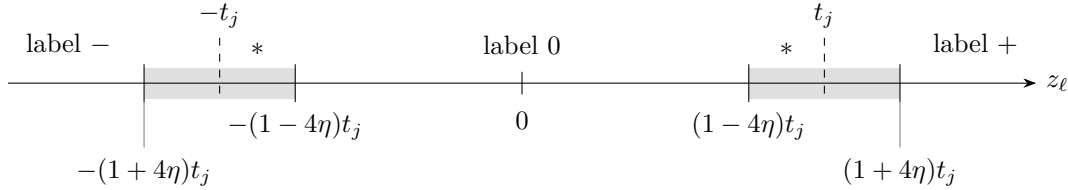

For a current measure $P$, let $m(P)=\int u_\ell dP$, and assign the scores
\begin{equation}\label{eq:scores}
S_+(P)=m(P)\text{, }S_-(P)=-m(P)\text{, and }S_0(P)=t_j.
\end{equation}
The predicted label is a maximizer of these three scores, ties being broken by the fixed priority order $0,+,-$. Under the target measure $Q_f$, a stable true label (namely, $0,+,-$) has a score at least $4\eta t_j$ larger than either incorrect score. For example, if the true label is $+$, then $S_+(Q_f)-S_0(Q_f)=z_\ell-t_j\geq4\eta t_j$ and $S_+(Q_f)-S_-(Q_f)=2z_\ell\geq4\eta t_j$, the calculation for the label $-$ is the same with the signs reversed, and for the label $0$ both differences are at least $t_j-|z_\ell|\geq4\eta t_j$.

\subsection{The update rule and the common descent bound}

We process the thresholds in the order $H,H-1,\ldots,0$, and the sample positions in the order $1,\ldots,k$ within each threshold, starting with $P=\nu$. At a boundary label or a correct stable prediction we leave $P$ unchanged. At a stable mistake, let $c$ be the true label and $\widehat c$ the predicted label, put $\psi_+=u_\ell$, $\psi_-=-u_\ell$, and $\psi_0=0$, define
\begin{equation}\label{eq:loss}
h=t_j(\psi_{\widehat c}-\psi_c),
\end{equation}
and apply \eqref{eq:update} with
\begin{equation}\label{eq:stepsize}
\lambda=\frac{\eta}{4L^2}.
\end{equation}
We emphasize that the predictor is never restarted when the threshold changes. The measure $Q_f$ is used to analyze the updates, but it is not an input to the update rule, which needs only the sampled coordinate, the threshold, the true label, and the current prediction.

Let us check that the constant score $S_0=t_j$ is accounted for by this rule. Write $S_c(P)=\int\psi_cdP+b_c$, where $b_0=t_j$ and $b_+=b_-=0$. The choice made by the predictor means that
$$
\int(\psi_{\widehat c}-\psi_c)dP
\geq b_c-b_{\widehat c},
$$
while the stable margin under $Q_f$ means that
$$
\int(\psi_{\widehat c}-\psi_c)dQ_f
\leq b_c-b_{\widehat c}-4\eta t_j.
$$
Subtracting cancels the constant additive term, and multiplying by $t_j$, we obtain
\begin{equation}\label{eq:loss-gap}
\int hdP-\int hdQ_f\geq4\eta t_j^2.
\end{equation}
Also $|h|\leq2Bt_j$, and Lemma~\ref{lem:one-step} therefore shows that
$$
D(Q_f\|P^+)-D(Q_f\|P)
\leq-4\eta\lambda t_j^2+2\lambda^2B^2t_j^2.
$$
Since $B^2=2L^2$, the choice \eqref{eq:stepsize} makes the coefficient on the right
$$
-4\eta\lambda+2\lambda^2B^2
=-\frac{\eta^2}{L^2}+\frac{\eta^2}{4L^2}
=-\frac{3\eta^2}{4L^2},
$$
so that
\begin{equation}\label{eq:strict-drop}
D(Q_f\|P^+)-D(Q_f\|P)
\leq-\frac{\eta^2}{2L^2}t_j^2.
\end{equation}
Thus a stable mistake forces an actual decrease of the nonnegative potential $D(Q_f\|P)$.

\begin{proposition}\label{prop:mistakes}
Let $d_j^{(\sigma)}(f)$ denote the number of stable mistakes at threshold $t_j$ in the procedure $\sigma$. Then
\begin{equation}\label{eq:weighted-mistakes}
\sum_{\sigma=1}^2\sum_{j=0}^H t_j^2d_j^{(\sigma)}(f)
\leq C_\eta L^2\log(e+\rho)
\end{equation}
for every $f\in\mathcal F_\rho$.
\end{proposition}

\begin{proof}
For one procedure we sum \eqref{eq:strict-drop} over all its mistakes, the potential being unchanged at every other query. The initial value is $D(Q_f\|\nu)$ and the terminal value is nonnegative, so
$$
\frac{\eta^2}{2L^2}\sum_{j=0}^H t_j^2d_j(f)
\leq D(Q_f\|\nu)
\leq C\log(e+\rho).
$$
We run the same argument for the other procedure and add the two bounds. The finiteness of all the intermediate entropies follows by applying the bounded density estimate of Lemma~\ref{lem:one-step} successively, starting from Lemma~\ref{lem:comparator}, and there are only $k(H+1)$ queries in each procedure. This is the common descent principle of \cite{BIK1,BIK2} in the form needed here, in which the initial potential is charged only once, even though the procedure passes through all the thresholds.
\end{proof}

The weights $t_j^2$ are the reason why the estimate remains useful at very small thresholds. There can be many small corrections, but their total squared size is bounded, and an estimate for the unweighted number of corrections would discard this information.

\section{Recovering approximate squares from correction histories}

For a function $f\in\mathcal F_\rho$ we record, in the procedure $\sigma$ and at the threshold $t_j$, the set
\begin{equation}\label{eq:records}
C_j^{(\sigma)}(f)=\{(\ell,c_f^{(\sigma)}(\ell,j)):
\text{a stable mistake occurs at }(\ell,j)\}.
\end{equation}
The history contains one of three stable labels at each recorded position, so that if it has $d$ entries, there are at most
\begin{equation}\label{eq:record-count}
\binom{k}{d}3^d
\end{equation}
possible histories. The sampled indices, the sequence of thresholds, the prior, the step size, and the rule for breaking ties are known to both the encoder and the decoder, and they are not part of the history.

The decoder replays the predictor. At a recorded position, it outputs the stored label and makes the update specified by that label and its current prediction, while at an unrecorded position it outputs its own predicted label and makes no update. The decoder does not need to know whether an unrecorded query was a correct stable prediction or a boundary query, since both leave the state of the encoder unchanged.

To see explicitly that no other information about $f$ is hidden in the state, suppose that the recorded updates encountered so far have losses $h_1,\ldots,h_m$. Repeated use of \eqref{eq:update} shows that
\begin{equation}\label{eq:replay-density}
\frac{dP}{d\nu}(e)=
\frac{\exp(-\lambda\sum_{q=1}^m h_q(e))}
{\int\exp(-\lambda\sum_{q=1}^m h_q)d\nu},
\end{equation}
the empty sum corresponding to $P=\nu$. At the next recorded query the decoder can calculate its prediction from this measure, form the next loss from the known coordinate and the stored label, and continue. This prescription is defined for every history, even one which did not arise from a target function, and if a stored label happens to equal the current prediction, its loss is zero, so including these additional histories only enlarges the finite family used in the probability estimate. Figure~\ref{fig:records} shows the order of the queries and a typical collection of histories.

\begin{lemma}\label{lem:reconstruction}
For a history generated by $f$, the decoder and the encoder have the same predictor state before every query, and the decoder therefore reproduces every stable target label correctly.
\end{lemma}

\begin{proof}
The encoder and decoder start from the same prior; suppose that their states agree before a query. At a recorded position they have the same prediction, read the same true label, and make the same update, while at an unrecorded position neither makes an update, so the states agree before the next query, and induction proves the first assertion. An unrecorded stable label was correctly predicted by the encoder, and hence also by the decoder, and at a recorded stable label the decoder outputs the stored answer.
\end{proof}

Let $\widehat c^{(\sigma)}(\ell,j)$ be the output of the decoder. For a sample position $\ell$ we look for the largest threshold with a nonzero output, setting $v_\ell^{(\sigma)}(f)=t_{j_*}^2$ if
\begin{equation}\label{eq:quantized-square}
j_*=\max\{j:\widehat c^{(\sigma)}(\ell,j)\in\{+,-\}\}
\end{equation}
exists, and $v_\ell^{(\sigma)}(f)=0$ otherwise, and we put
\begin{equation}\label{eq:complex-square}
v_\ell(f)=v_\ell^{(1)}(f)+v_\ell^{(2)}(f).
\end{equation}
The signs are used to guide the descent, but only the fact that they are nonzero is used to approximate a square.

\begin{lemma}\label{lem:quantization}
There is an absolute constant $C$ such that, for every $f\in\mathcal F_\rho$ and every sample position $\ell$,
\begin{equation}\label{eq:pointwise-error}
|v_\ell(f)-|a_{J_\ell}(f)|^2|
\leq C\eta|a_{J_\ell}(f)|^2+C\tau^2.
\end{equation}
Consequently,
\begin{equation}\label{eq:total-error}
\sum_{\ell=1}^k|v_\ell(f)-|a_{J_\ell}(f)|^2|
\leq C\eta\sum_{\ell=1}^k|a_{J_\ell}(f)|^2+Ck\tau^2.
\end{equation}
\end{lemma}

\begin{proof}
Consider one real target value $z$. Stable reconstruction forces a nonzero decoded label whenever $t_j\leq\frac{|z|}{1+4\eta}$, and forces a zero label whenever $t_j\geq\frac{|z|}{1-4\eta}$, and the top threshold always has a stable zero label because $|z|\leq L$.

Suppose first that some threshold satisfies the first inequality, and let $t_r$ be the largest such threshold. Then $r<H$ and
$$
t_r\leq\frac{|z|}{1+4\eta}<t_{r+1}=(1+\eta)t_r,
$$
so the largest nonzero decoded threshold exists and satisfies
$$
\frac{|z|}{(1+\eta)(1+4\eta)}<t_r\leq t_{j_*}
<\frac{|z|}{1-4\eta}.
$$
For $0<\eta<10^{-3}$, the mean value theorem shows that
$$
\frac{1}{(1+\eta)^2(1+4\eta)^2}\geq1-12\eta
$$
and that
$$
\frac{1}{(1-4\eta)^2}\leq1+12\eta,
$$
so squaring the preceding bounds for the threshold, we find that $|t_{j_*}^2-z^2|\leq12\eta z^2$.

If no threshold satisfies the first inequality, then $|z|<(1+4\eta)\tau$, and any nonzero decoded threshold must still be smaller than $\frac{|z|}{1-4\eta}$, and hence at most $C\tau$, so both the target square and its decoded approximation are at most $C\tau^2$. This proves the one-dimensional estimate in all cases, including the case in which every output is zero. We apply it to the real and imaginary parts and add, and summation over $\ell$ proves \eqref{eq:total-error}.
\end{proof}

For $\sigma=1,2$ and $0\leq j\leq H$, we define the decoded band
\begin{equation}\label{eq:bands}
A_j^{(\sigma)}(f)=\{\ell:v_\ell^{(\sigma)}(f)=t_j^2\}.
\end{equation}
For each fixed $\sigma$ the bands are disjoint, and
\begin{equation}\label{eq:band-expansion}
v_\ell(f)=\sum_{\sigma=1}^2\sum_{j=0}^H
t_j^21_{A_j^{(\sigma)}(f)}(\ell).
\end{equation}

\begin{lemma}\label{lem:tail}
For fixed sampled indices, $\sigma$, and $j$, every choice of histories $C_H^{(\sigma)},C_{H-1}^{(\sigma)},\ldots,C_j^{(\sigma)}$ determines a subset of $[k]$, independently of how the histories at the lower levels are completed, and when the histories come from $f$, this subset is $A_j^{(\sigma)}(f)$.
\end{lemma}

\begin{proof}
Replay the decoder from its initial prior through the level $j$, and form the set
$$
\{\ell:\widehat c^{(\sigma)}(\ell,j)\in\{+,-\}\text{ and }
\widehat c^{(\sigma)}(\ell,r)=0\text{ for every }r>j\}.
$$
All these queries are processed before the lower levels, so the set is determined by the given histories, and any completion at the lower levels leaves these outputs unchanged. By the definition of the largest nonzero threshold, this is exactly the band at the level $j$ for any such completion, and in particular, it is $A_j^{(\sigma)}(f)$ for a history generated by $f$.
\end{proof}

This is the only property of the code concerning dependence which the probability argument will need. We do not require a band to be determined by the history at its own level alone, and all the histories at the higher levels are included.

\begin{figure}[htbp]
\centering
\begin{tikzpicture}[x=0.95cm,y=0.72cm,font=\small,>={Stealth}]
 \fill[blue!10] (0.5,1.5) rectangle (8.5,5.5);
 \foreach \y in {0,...,5}
  {\foreach \x in {1,...,8} {\draw[gray] (\x,\y) circle (2pt);}}
 \foreach \x/\y in {2/5,6/5,3/4,1/3,7/3,4/2,5/1,8/1,2/0,6/0,7/0}
  {\fill (\x,\y) circle (3pt);}
 \foreach \y in {0,...,5} {\node[anchor=east] at (0.3,\y) {$t_{\y}$};}
 \foreach \x in {1,...,8} {\node[below] at (\x,-0.35) {$\x$};}
 \node at (4.5,-1.35) {sample position $\ell$};
 \draw[->,thick] (9.3,5.3)--(9.3,-0.3);
 \node[anchor=west,align=left] at (9.45,2.5) {order of\\processing};
 \node at (4.5,6.15) {the histories in the shaded rows determine the band $A_2$};
\end{tikzpicture}
\caption{The order in which the queries are processed, drawn with $H=5$ and $k=8$. The thresholds are processed from the top row to the bottom row, and the positions from left to right within each row. Filled circles mark stable mistakes, whose positions and labels form the histories, and the decoded labels in the shaded rows, and hence the band $A_2$, are determined by the histories at the levels $5$, $4$, $3$, and $2$, whatever happens at the lower levels.}
\label{fig:records}
\end{figure}
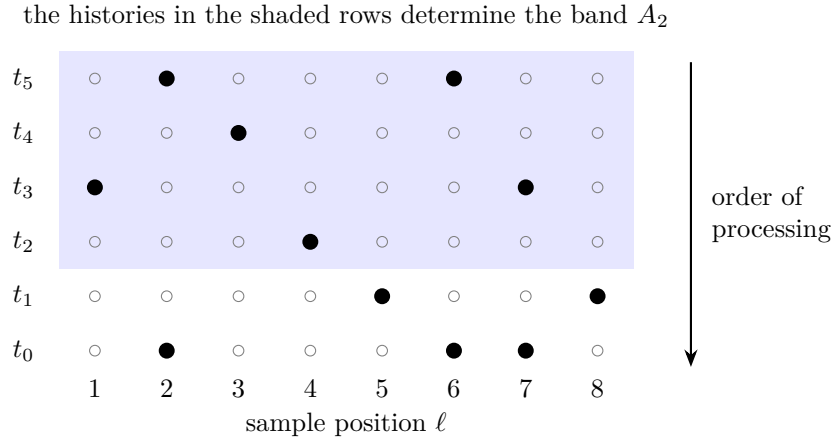

\section{The weighted cost of the histories}

We next assign a positive weight to each possible history. The weights play the role of probabilities in a union bound, although their total mass is allowed to be less than one. For a history $C$ with $d$ entries we put
\begin{equation}\label{eq:record-weight}
\pi_j(C)=\frac{2^{-(d+1)}}{\binom{k}{d}3^d},
\end{equation}
the subscript recording the threshold, although the formula is the same at every threshold. By \eqref{eq:record-count},
\begin{equation}\label{eq:kraft}
\sum_C\pi_j(C)\leq\sum_{d=0}^k2^{-(d+1)}\leq1.
\end{equation}
For the tuple $C_{\geq j}=(C_j,\ldots,C_H)$ of histories at the thresholds $j,\ldots,H$, we define
$$
\pi_{\geq j}(C_{\geq j})=\prod_{r=j}^H\pi_r(C_r),
$$
and the sum of these products over all such sequences is also at most one.

We let $w_j=\frac{t_j^2}{T}$ and
\begin{equation}\label{eq:weights-q}
q_j^{(\sigma)}(f)=\log\left(
\frac{8}{w_j\pi_{\geq j}(C_{\geq j}^{(\sigma)}(f))}\right).
\end{equation}
The quantity $q_j^{(\sigma)}$ is the cost assigned to the information which determines the band $A_j^{(\sigma)}$, and it includes the histories at the higher levels because Lemma~\ref{lem:tail} requires them.

\begin{proposition}\label{prop:weighted-cost}
For a sufficiently large constant $C_\eta$, put
\begin{equation}\label{eq:Krho}
K_\rho=C_\eta L^2\log(2k)\log(e+\rho).
\end{equation}
Then $2T\leq K_\rho$, and every $f\in\mathcal F_\rho$ satisfies
\begin{equation}\label{eq:cost-bound}
\sum_{\sigma=1}^2\sum_{j=0}^H t_j^2q_j^{(\sigma)}(f)\leq K_\rho.
\end{equation}
\end{proposition}

\begin{proof}
For a history with $d$ entries,
$$
\log\left(\frac{1}{\pi_j(C)}\right)
=(d+1)\log2+\log\binom{k}{d}+d\log3
\leq\log2+d\log(6k),
$$
where we used the elementary bound $\binom{k}{d}\leq k^d$, which also holds for $d=0$.

Fix one procedure. Since all the sums are finite, we may reverse their order, and we find that
$$
\sum_{j=0}^H t_j^2
\log\left(\frac{1}{\pi_{\geq j}(C_{\geq j})}\right)
=\sum_{j=0}^H t_j^2\sum_{r=j}^H
\log\left(\frac{1}{\pi_r(C_r)}\right)
$$
equals
$$
\sum_{r=0}^H\log\left(\frac{1}{\pi_r(C_r)}\right)
\sum_{j=0}^r t_j^2,
$$
which, by \eqref{eq:geometric-sums}, is at most
$$
C_\eta\left(T+\log(2k)\sum_{r=0}^H t_r^2d_r\right).
$$
This calculation is where the order of processing pays off. A history at the level $r$ is charged at every level $j\leq r$, but the sum of those charges is at most a constant times its own weight $t_r^2$.

We must also bound the contribution of $w_j$ in \eqref{eq:weights-q}. Put $a=(1+\eta)^{-2}$, which lies in $(0,1)$. Writing $j=H-m$, we have
$$
w_{H-m}=\frac{a^m}{\sum_{r=0}^H a^r}
$$
for $0\leq m\leq H$, where the denominator lies between $1$ and $\frac{1}{1-a}$, and moreover
$$
\sum_{m=0}^H mw_{H-m}\leq\sum_{m=0}^\infty ma^m
=\frac{a}{(1-a)^2}.
$$
Consequently,
$$
\sum_{j=0}^H w_j\log\left(\frac{1}{w_j}\right)
\leq\log\left(\frac{1}{1-a}\right)
+\log\left(\frac{1}{a}\right)\frac{a}{(1-a)^2}
\leq C_\eta,
$$
and multiplying by $T$, we obtain a bound for $\sum_jt_j^2\log(\frac{1}{w_j})$. Adding the two procedures and the terms $t_j^2\log8$, and using Proposition~\ref{prop:mistakes}, we find that
$$
\sum_{\sigma,j}t_j^2q_j^{(\sigma)}
\leq C_\eta(T+L^2\log(2k)\log(e+\rho)).
$$
Since $T\leq C_\eta L^2$, this is \eqref{eq:cost-bound}, and enlarging the same constant in \eqref{eq:Krho} also ensures that $2T\leq K_\rho$.
\end{proof}

\section{A conditional Rademacher estimate}

Let $\varepsilon_1,\ldots,\varepsilon_k$ be independent random variables taking the values $1$ and $-1$ with equal probability, independently of the sampled indices, the familiar Rademacher variables. For the moment the sampled indices are still fixed, and we set
\begin{equation}\label{eq:Vrho}
V_\rho=\sup_{f\in\mathcal F_\rho}\sum_{\ell=1}^k v_\ell(f).
\end{equation}

\begin{proposition}\label{prop:rademacher}
Conditionally on $J_1,\ldots,J_k$,
\begin{equation}\label{eq:coded-rademacher}
\mathbb E_\varepsilon\sup_{f\in\mathcal F_\rho}
\left|\sum_{\ell=1}^k\varepsilon_\ell v_\ell(f)\right|
\leq C\sqrt{K_\rho V_\rho}.
\end{equation}
\end{proposition}

\begin{proof}
For a fixed subset $A\subset[k]$, independence shows that
$$
\mathbb E_\varepsilon\exp\left(v\sum_{\ell\in A}\varepsilon_\ell\right)
=(\cosh v)^{|A|}\leq\exp\left(\frac{v^2|A|}{2}\right),
$$
the last inequality following from Lemma~\ref{lem:exponential} applied to one sign. If $A$ is nonempty and $a>0$, the exponential Markov inequality shows that
$$
\mathbb P_\varepsilon\left\{\sum_{\ell\in A}\varepsilon_\ell>a\right\}
\leq\exp\left(-va+\frac{v^2|A|}{2}\right)
$$
for every $v>0$. Taking $v=\frac{a}{|A|}$, applying the same argument to the negative sum, and then setting $a=\sqrt{2|A|b}$, we obtain
\begin{equation}\label{eq:sign-tail}
\mathbb P_\varepsilon\left\{
\left|\sum_{\ell\in A}\varepsilon_\ell\right|>
\sqrt{2|A|b}\right\}\leq2e^{-b}
\end{equation}
for every $b>0$, and when $A$ is empty the event is empty, so the inequality still holds.

Fix $u\geq0$. For a procedure $\sigma$, a level $j$, and any possible sequence of histories at the levels $j,\ldots,H$, Lemma~\ref{lem:tail} determines a fixed band $A$, and we apply \eqref{eq:sign-tail} with
$$
b=\log\left(\frac{8}{w_j\pi_{\geq j}(C_{\geq j})}\right)+u.
$$
The sum of the resulting exceptional probabilities, over both procedures, all the levels, and all these sequences of histories, is at most
$$
\sum_{\sigma=1}^2\sum_{j=0}^H
\frac{w_je^{-u}}{4}\sum_{C_{\geq j}}\pi_{\geq j}(C_{\geq j})
\leq\frac{1}{2}e^{-u},
$$
so outside an event of probability at most $\frac{1}{2}e^{-u}$, all the required estimates hold simultaneously, and in particular, they hold for the histories generated by every $f\in\mathcal F_\rho$. At this stage we used no bound on the number of corrections in an arbitrary history, since the total weight bound \eqref{eq:kraft} controls the union over all histories, and the bound involving $K_\rho$ will now be applied only to histories which actually come from a function in $\mathcal F_\rho$.

On this event, \eqref{eq:band-expansion} and the Cauchy--Schwarz inequality show that
$$
\left|\sum_{\ell=1}^k\varepsilon_\ell v_\ell(f)\right|
\leq\sum_{\sigma,j}t_j^2\sqrt{2|A_j^{(\sigma)}(f)|(q_j^{(\sigma)}(f)+u)},
$$
and the right side is at most
$$
\left(2\sum_{\sigma,j}t_j^2|A_j^{(\sigma)}(f)|\right)^{\frac{1}{2}}
\left(\sum_{\sigma,j}t_j^2(q_j^{(\sigma)}(f)+u)\right)^{\frac{1}{2}}
\leq\sqrt{2V_\rho(K_\rho+2uT)},
$$
where in the last step we used the identity
$$
\sum_{\sigma,j}t_j^2|A_j^{(\sigma)}(f)|=\sum_{\ell=1}^k v_\ell(f),
$$
the identity $\sum_{\sigma,j}t_j^2=2T$, and Proposition~\ref{prop:weighted-cost}. The bound is therefore uniform in $f$. No independence between different bands, or between the two procedures, is needed, because the single union bound already covers all of them.

Let $X$ denote the supremum on the left side of \eqref{eq:coded-rademacher}. If $V_\rho=0$, all the $v_\ell(f)$ vanish and the result is immediate. Otherwise, since $2T\leq K_\rho$, the preceding estimate implies that
$$
\mathbb P_\varepsilon\{X>A\sqrt{1+u}\}\leq\frac{1}{2}e^{-u},
$$
where $A=\sqrt{2V_\rho K_\rho}$, and integrating the tail, with the substitution $x=A\sqrt{1+u}$ for $x\geq A$, we find that
$$
\mathbb E_\varepsilon X
\leq A+\frac{A}{4}\int_0^\infty\frac{e^{-u}}{\sqrt{1+u}}du
\leq\frac{5A}{4},
$$
which proves \eqref{eq:coded-rademacher}.
\end{proof}

The conditioning is essential. The histories depend on the sampled indices, but once those indices are fixed, every possible history determines a fixed band and the signs remain independent. We shall apply symmetrization to the original coefficient squares before replacing them by their decoded approximations, and there is therefore no assumption that the code is independent of the sample.

\section{Uniform control of the empirical second moment}

We now return to the randomness of the indices, and define
\begin{equation}\label{eq:Drho}
\mathcal D_\rho=
\mathbb E_J\sup_{f\in\mathcal F_\rho}
\left|\sum_{\ell=1}^k|a_{J_\ell}(f)|^2-\delta\rho^2\right|.
\end{equation}
The following proposition is the quantitative estimate needed for the main theorem.

\begin{proposition}\label{prop:empirical}
For every $0<\epsilon<1$ there is a constant $C_\epsilon$ such that, if $\mathcal F_\rho$ is nonempty, $0<\rho\leq L\sqrt n$, and
\begin{equation}\label{eq:fixed-point-condition}
\rho^2\geq C_\epsilon\frac{n}{k}L^2\log(2k)\log(e+\rho),
\end{equation}
then
\begin{equation}\label{eq:empirical-small}
\mathcal D_\rho\leq\epsilon\delta\rho^2.
\end{equation}
\end{proposition}

\begin{proof}
Fix $0<\epsilon<1$. For the moment $\eta$ is any number in $(0,10^{-3})$, and we shall choose it depending only on $\epsilon$ at the end. All the suprema in this proof are over $\mathcal F_\rho$.

We first spell out the symmetrization. Let $J'_1,\ldots,J'_k$ be an independent copy of the sample, write $h_i(f)=|a_i(f)|^2$, and for a fixed original sample set $Z_f=\sum_\ell(h_{J_\ell}(f)-h_{J'_\ell}(f))$. Since $\mathbb E_{J'}\sum_\ell h_{J'_\ell}(f)=\delta\rho^2$,
$$
\sup_f|\mathbb E_{J'}Z_f|
\leq\sup_f\mathbb E_{J'}|Z_f|
\leq\mathbb E_{J'}\sup_f|Z_f|,
$$
and averaging over the original sample, we obtain
$$
\mathcal D_\rho\leq
\mathbb E_{J,J'}\sup_f
\left|\sum_{\ell=1}^k(h_{J_\ell}(f)-h_{J'_\ell}(f))\right|.
$$
Each pair $(J_\ell,J'_\ell)$ has the same distribution when its entries are interchanged, and the pairs are independent, so we may interchange any chosen collection of pairs without changing their joint distribution. Introducing independent signs, we see that the right side equals
$$
\mathbb E_{J,J',\varepsilon}\sup_f\left|\sum_{\ell=1}^k
\varepsilon_\ell(h_{J_\ell}(f)-h_{J'_\ell}(f))\right|,
$$
and the triangle inequality, together with the identical distributions of the two samples, shows that this is at most twice the corresponding expression for one sample. Hence
\begin{equation}\label{eq:symmetrization}
\mathcal D_\rho\leq
2\mathbb E_{J,\varepsilon}\sup_{f\in\mathcal F_\rho}
\left|\sum_{\ell=1}^k\varepsilon_\ell|a_{J_\ell}(f)|^2\right|.
\end{equation}

For a fixed sample we put
$$
M_\rho(J)=\sup_{f\in\mathcal F_\rho}
\sum_{\ell=1}^k|a_{J_\ell}(f)|^2.
$$
For each sample, $M_\rho(J)$ is at most $\delta\rho^2$ plus the supremum of the absolute deviation in \eqref{eq:Drho}, and hence
\begin{equation}\label{eq:M-mean}
\mathbb E_J M_\rho(J)\leq\mathcal D_\rho+\delta\rho^2.
\end{equation}
Lemma~\ref{lem:quantization} shows that, for every sequence of signs,
$$
\sup_f\left|\sum_{\ell=1}^k\varepsilon_\ell|a_{J_\ell}(f)|^2\right|
\leq\sup_f\left|\sum_{\ell=1}^k\varepsilon_\ell v_\ell(f)\right|
+C\eta M_\rho(J)+Ck\tau^2,
$$
and also that
\begin{equation}\label{eq:V-upper}
V_\rho\leq(1+C\eta)M_\rho(J)+Ck\tau^2.
\end{equation}
We apply Proposition~\ref{prop:rademacher} conditionally on the sample and then average, bounding the term with the square root by
$$
\mathbb E_J\sqrt{K_\rho V_\rho}
\leq\sqrt{K_\rho\mathbb E_J V_\rho},
$$
where $K_\rho$ does not depend on the sample. Using $k\tau^2=\eta\delta\rho^2$ and \eqref{eq:M-mean}, we conclude from \eqref{eq:symmetrization} that
\begin{equation}\label{eq:self-bound}
\mathcal D_\rho
\leq A\sqrt{K_\rho(\mathcal D_\rho+\delta\rho^2)}
+A\eta(\mathcal D_\rho+\delta\rho^2),
\end{equation}
where $A$ is an absolute constant, all the dependence on $\eta$ in the first term being contained in $K_\rho$.

We now choose the constants in an order which makes the prescribed accuracy explicit. We take $\eta$ so small that $0<\eta<10^{-3}$, $A\eta\leq\frac{1}{4}$, and $2A\eta\leq\frac{\epsilon}{2}$, a choice which fixes the constant $C_\eta$ in $K_\rho$. Let $b=\delta\rho^2$. Using $\sqrt{x+y}\leq\sqrt x+\sqrt y$ and
$$
A\sqrt{K_\rho\mathcal D_\rho}
\leq\frac{1}{4}\mathcal D_\rho+A^2K_\rho,
$$
we can absorb the two terms containing $\mathcal D_\rho$ on the right side of \eqref{eq:self-bound}, and we obtain
$$
\mathcal D_\rho\leq2A^2K_\rho+2A\sqrt{K_\rho b}+2A\eta b.
$$
We choose $\kappa>0$, depending only on $\epsilon$, so small that
$$
2A^2\kappa+2A\sqrt\kappa\leq\frac{\epsilon}{2}.
$$
If $K_\rho\leq\kappa b$, the last two displays show that $\mathcal D_\rho\leq\epsilon b$, and by the definition of $K_\rho$ this requirement is exactly \eqref{eq:fixed-point-condition}, once $C_\epsilon$ is chosen sufficiently large. This proves the proposition.
\end{proof}

Let us record why the suprema used above are well defined. The space $E$ is finite-dimensional, and the $L_1$ unit ball is closed and, by \eqref{eq:trivial-bound}, bounded in $L_2$, so it is compact, and $\mathcal F_\rho$ is a closed subset of this ball. Coefficient evaluation is continuous, and hence the original suprema are finite and attained. A decoded map need not be continuous, but for a fixed sample its range is finite, because there are only finitely many histories. Finally, the samples of indices and the sequences of signs range over finite sets, so all the expectations in the proof are expectations of well-defined functions on finite probability spaces.

\section{Choosing the radius and proving the selection theorem}

We must solve the inequality in Proposition~\ref{prop:empirical}, which has the familiar feature that the unknown radius appears inside a logarithm as well as outside it, and the next calculation isolates this harmless dependence.

\begin{lemma}\label{lem:logarithm}
For every $A_0>0$ there is a constant $C_0\geq1$ such that, for every $x>0$, the number
$$
R=C_0x\sqrt{\log(e+x)}
$$
satisfies
$$
R^2\geq A_0x^2\log(e+R).
$$
\end{lemma}

\begin{proof}
Put $b=\log(e+x)$, so that $b\geq1$. Then $e+C_0x\sqrt b\leq C_0(e+x)\sqrt b$, and therefore
$$
\log(e+R)\leq\log C_0+b+\frac{1}{2}\log b
\leq\left(\log C_0+\frac{3}{2}\right)b.
$$
We choose $C_0$ so large that $C_0^2\geq A_0(\log C_0+\frac{3}{2})$, which is possible because a square grows faster than a logarithm, and since $R^2=C_0^2x^2b$, the conclusion follows.
\end{proof}

\begin{proof}[Proof of the first assertion of Theorem~\ref{thm:main}]
We apply Lemma~\ref{lem:logarithm} with $x=M_k$ and with $A_0$ equal to the constant of Proposition~\ref{prop:empirical} for $\epsilon=\frac{1}{100}$, and set
\begin{equation}\label{eq:chosen-radius}
\rho=C_0M_k\sqrt{\log(e+M_k)}.
\end{equation}
If $\rho\geq L\sqrt n$, the elementary estimate \eqref{eq:trivial-bound} proves the result on all of $E$, so any set with at least $n-k$ indices suffices, and if $\mathcal F_\rho$ is empty, the observation following \eqref{eq:sphere} leads to the same conclusion. We may therefore suppose that $\rho<L\sqrt n$ and that $\mathcal F_\rho$ is nonempty.

The chosen radius satisfies \eqref{eq:fixed-point-condition}, and hence $\mathcal D_\rho\leq\frac{1}{100}\delta\rho^2$, so Markov's inequality shows that
\begin{equation}\label{eq:good-event}
\mathbb P_J\left\{
\sup_{f\in\mathcal F_\rho}
|\|\Gamma f\|_2^2-\delta\rho^2|
\leq\frac{1}{3}\delta\rho^2
\right\}\geq\frac{97}{100}.
\end{equation}
Choose any realization in this event. Its set $S$ of distinct sampled indices has at most $k$ elements, so that $I=[n]\setminus S$ has at least $n-k$ elements, and Lemma~\ref{lem:radius}, with $c_0=\frac{1}{3}$, shows that $\|f\|_2\leq\rho\|f\|_1$ for every $f\in E_I$. Substituting \eqref{eq:chosen-radius}, we obtain \eqref{eq:main-bound}.
\end{proof}

For the second assertion we want the set of distinct sampled indices to occupy approximately one half of $[n]$, and sampling with replacement approximately $n\log2$ times has this effect. We include the elementary variance calculation needed to make this precise.

\begin{lemma}\label{lem:occupancy}
Let $n\geq4$ and $k=\lfloor n\log2\rfloor$. If $D$ is the number of distinct indices in $k$ independent uniform samples from $[n]$, then
$$
\mathbb P\left\{\left|D-\frac{n}{2}\right|\leq C\sqrt n\right\}
\geq\frac{3}{4}
$$
for an absolute constant $C$.
\end{lemma}

\begin{proof}
Let $X_i$ be the indicator of the event that the index $i$ is sampled, so that $D=\sum_iX_i$ and
$$
\mathbb E D=n\left(1-\left(1-\frac{1}{n}\right)^k\right).
$$
To see how close this mean is to $\frac{n}{2}$, observe that, for $0\leq x\leq\frac{1}{4}$,
$$
|\log(1-x)+x|
=\int_0^x\frac{t}{1-t}dt
\leq\frac{x^2}{2(1-x)}\leq x^2.
$$
Since $0\leq n\log2-k<1$ and $k<n$, it follows that
$$
\left|k\log\left(1-\frac{1}{n}\right)+\log2\right|
\leq\frac{1}{n}+\frac{k}{n^2}<\frac{2}{n},
$$
so that $(1-\frac{1}{n})^k=\frac{1}{2}e^u$ for some $u$ with $|u|\leq\frac{2}{n}\leq\frac{1}{2}$. The mean value theorem shows that $|e^u-1|\leq2|u|$, and hence
$$
\left|\left(1-\frac{1}{n}\right)^k-\frac{1}{2}\right|
\leq\frac{2}{n}
$$
and $|\mathbb E D-\frac{n}{2}|\leq2$.

For distinct indices $i$ and $j$, the probability that both are absent is $(1-\frac{2}{n})^k$, so inclusion--exclusion shows that
$$
\mathbb E(X_iX_j)=1-2\left(1-\frac{1}{n}\right)^k
+\left(1-\frac{2}{n}\right)^k,
$$
and subtracting $\mathbb E X_i\mathbb E X_j$, we find that
$$
\operatorname{Cov}(X_i,X_j)
=\left(1-\frac{2}{n}\right)^k
-\left(1-\frac{1}{n}\right)^{2k}\leq0,
$$
the inequality following from $1-\frac{2}{n}\leq(1-\frac{1}{n})^2$. Therefore
$$
\operatorname{Var}(D)
\leq\sum_{i=1}^n\operatorname{Var}(X_i)\leq\frac{n}{4},
$$
and Chebyshev's inequality shows that $\mathbb P\{|D-\mathbb ED|>\sqrt n\}\leq\frac{1}{4}$. On the complementary event,
$$
\left|D-\frac{n}{2}\right|\leq\sqrt n+2\leq2\sqrt n,
$$
since $n\geq4$, and this proves the lemma with $C=2$.
\end{proof}

\begin{proof}[Proof of the second assertion of Theorem~\ref{thm:main}]
For $n=2$ the assertion follows from \eqref{eq:trivial-bound} after the constant is enlarged. Suppose that $n\geq4$ is even, and set $k=\lfloor n\log2\rfloor$. Then $2\leq k<n$ and
$$
\delta=\frac{k}{n}\leq\log2<\frac{3}{4}.
$$
Also $k\geq\frac{n}{2}$, so $\log(2k)$ is comparable to $\log(2n)$, with absolute constants, and in particular,
$$
M_k\leq CL\sqrt{\log(2n)}.
$$
We choose the radius \eqref{eq:chosen-radius} as above. If the trivial estimate applies, or if $\mathcal F_\rho$ is empty, both spans of any partition satisfy the desired bound, and we choose a partition into two sets of equal size.

Otherwise the event \eqref{eq:good-event} has probability at least $\frac{97}{100}$ and the event in Lemma~\ref{lem:occupancy} has probability at least $\frac{3}{4}$, so a union bound for their complements shows that their intersection has probability at least $1-\frac{3}{100}-\frac{1}{4}=\frac{18}{25}>0$, no independence between these events being required. We choose a sample in the intersection and let $S$ be its set of distinct indices, so that both $S$ and its complement have cardinality within $C\sqrt n$ of $\frac{n}{2}$.

The first part of Lemma~\ref{lem:radius} controls the span of the complement, and for the selected span its additional hypothesis holds because $(1+\frac{1}{3})\delta<1$. Both spans therefore satisfy the norm comparison with constant $\rho$, and the bound on $M_k$ and the definition of $\rho$ prove \eqref{eq:half-bound}.
\end{proof}

\section{Concentration and recovery from coefficient measurements}\label{sec:recovery}

The norm comparison prevents a vector in the selected span from placing too much of its mass on a small set, and this provides another point of contact with the companion papers. The first paper \cite{BIK1} applies $\Lambda(p)$ selection to recovery after the loss of Fourier data, and the second \cite{BIK2} obtains recovery from the restricted isometry of the measured rows. Here the selected large span is the kernel of the measurements, and the comparison of the $L_1$ and $L_2$ norms prevents a vector in the kernel from being too concentrated. This is the null-space mechanism used in \cite{IKLM2024}, which complements the approach through uncertainty and restriction in \cite{IM2025}, and we give the argument in full to show precisely where the improved norm comparison enters.

\begin{corollary}\label{cor:concentration}
Suppose that a subspace $F\subset L_2(\mu)$ satisfies $\|f\|_2\leq R\|f\|_1$ for every $f\in F$. Then for every measurable $A\subset\Omega$ and every $f\in F$,
\begin{equation}\label{eq:mass-concentration}
\int_A|f|d\mu\leq R\sqrt{\mu(A)}\|f\|_1.
\end{equation}
In particular, a nonzero function in $F$ which vanishes almost everywhere outside $A$ must satisfy $\mu(A)\geq\frac{1}{R^2}$.
\end{corollary}

\begin{proof}
The Cauchy--Schwarz inequality shows that $\int_A|f|d\mu\leq\sqrt{\mu(A)}\|f\|_2$, and inserting the assumed norm comparison, we obtain \eqref{eq:mass-concentration}. If $f$ is supported on $A$, the left side equals $\|f\|_1$, and dividing by this positive number proves the final assertion.
\end{proof}

For an explicit recovery statement, let $\Omega$ consist of $n$ points, each of mass $\frac{1}{n}$, and let $(\phi_i)_{i=1}^n$ be a complete orthonormal system satisfying \eqref{eq:bounded-system}, as is the case, for example, for the characters of a finite abelian group of order $n$, with $L=1$. Let $I$ be the set provided by the first part of Theorem~\ref{thm:main}, put $S=[n]\setminus I$, and define
$$
\mathcal A f=(a_i(f))_{i\in S}.
$$
There are at most $k$ measured coefficients and $\ker\mathcal A=E_I$, and we set
\begin{equation}\label{eq:recovery-radius}
R=CM_k\sqrt{\log(e+M_k)},
\end{equation}
where the constant is large enough for the norm comparison on this kernel.

To interpret these measurements as Fourier data, take $\Omega=G$, where $G$ is a finite abelian group of order $n$, and index the orthonormal basis by its characters $\gamma\in\widehat G$. With the unitary normalization
$$
\widehat f(\gamma)=\frac{1}{\sqrt n}\sum_{x\in G}f(x)\overline{\gamma(x)},
$$
we have
$$
a_\gamma(f)=\frac{1}{n}\sum_{x\in G}f(x)\overline{\gamma(x)}
=\frac{1}{\sqrt n}\widehat f(\gamma),
$$
so knowing $\mathcal A f$ is equivalent to knowing $\widehat f$ on $S$. The set $I$ consists of the unobserved frequencies, and any two signals with the same observed data differ by a function whose Fourier transform is supported in $I$, which are exactly the differences controlled by the norm comparison. The sparsity assumption below concerns the support of the signal in $G$.

\begin{corollary}\label{cor:recovery}
In this finite setting, suppose that an integer $s\geq1$ satisfies
\begin{equation}\label{eq:sparsity-condition}
R\sqrt{\frac{s}{n}}<\frac{1}{2}.
\end{equation}
Then every function $f$ supported on at most $s$ points is the unique solution of the minimization problem
\begin{equation}\label{eq:minimization}
\min\{\|g\|_1:\mathcal A g=\mathcal A f\},
\end{equation}
taken over all complex-valued functions $g$ on $\Omega$, and the same set of measurements works for every such $f$. Here $\|g\|_1=\frac{1}{n}\sum_{x\in\Omega}|g(x)|$, and multiplying the objective by $n$ does not change its minimizers.
\end{corollary}

\begin{proof}
Let $T$ be the support of $f$, and let $h\neq0$ belong to $\ker\mathcal A$. Corollary~\ref{cor:concentration} shows that
$$
\|h1_T\|_1\leq R\sqrt{\frac{|T|}{n}}\|h\|_1
<\frac{1}{2}\|h\|_1,
$$
so that $\|h1_T\|_1<\|h1_{\Omega\setminus T}\|_1$. Every other feasible function has the form $f+h$ for such an $h$, and the triangle inequality shows that
$$
\|f+h\|_1
=\|(f+h)1_T\|_1+\|h1_{\Omega\setminus T}\|_1
\geq\|f\|_1-\|h1_T\|_1+\|h1_{\Omega\setminus T}\|_1
>\|f\|_1,
$$
so $f$ is the unique minimizer. The argument is uniform over all supports of cardinality at most $s$.
\end{proof}

Inserting \eqref{eq:Mk} into \eqref{eq:sparsity-condition}, we see that the conclusion holds whenever
\begin{equation}\label{eq:explicit-sparsity}
s<c\frac{k}{L^2\log(2k)\log(e+M_k)}
\end{equation}
for a sufficiently small absolute constant $c$. The measured set is chosen once and works for every signal with this sparsity, without prior knowledge of its support. For the character basis $L=1$, and if $k=\lfloor\alpha n\rfloor$ for a fixed $0<\alpha<1$, then for sufficiently large $n$ we obtain uniform exact recovery whenever
$$
s<c_\alpha\frac{n}{\log n\log\log n}.
$$
If we use \eqref{eq:old-factor} in the same null-space argument instead, we obtain a sufficient sparsity bound of order $\frac{n}{\log n(\log\log n)^5}$, so the improvement in the norm comparison increases this sufficient range for recovery by a factor of order $(\log\log n)^4$. This comparison concerns the guarantees obtained from the two norm bounds, and it does not identify the largest sparsity for which a given set of measurements permits recovery.

The relation to \cite{IKLM2024} is particularly close, since that paper also applies the estimates for proportional subsystems of \cite{GMPT2008} to the kernel of incomplete coefficient measurements. Its asymptotic theorem uses $o(n)$ sampled coefficients and has success probability tending to one, while our estimate \eqref{eq:explicit-sparsity} specifies the dependence on $k$, and the sampling argument above provides success probability at least $\frac{97}{100}$ for the set of distinct indices in $k$ independent uniform draws, so the two probability statements should be distinguished. We do not assert an improvement over the best restricted isometry bounds. The point is that the sharper norm comparison obtained by descent passes directly into the established mechanism of sparse recovery, and we refer also to \cite{GMPT2007,GMPT2008}.

\section{Further remarks}

The proof separates the two logarithms in the final estimate. The factor $\log(2k)$ comes from recording the positions and labels of the corrections, and the factor $\log(e+\rho)$ from the entropy of the representing measure $Q_f$, while the weighted summation of the histories at the higher levels contributes only a constant depending on the fixed accuracy parameter, so in particular, the number of thresholds contributes no additional logarithm.

The descent calculation itself is short. Once \eqref{eq:loss-gap} is available, the potential decreases by a multiple of $t_j^2$ at every mistake, and the additional work is needed to retain the information in those weights when passing from corrections to a probability estimate. The reconstruction by the largest nonzero threshold, and the dependence of each band only on its own level and the higher levels, are what make this possible.

Our comparison with \cite{GMPT2008} concerns its theorem about proportional subsets, and we do not improve the general majorizing-measure theorem proved there. The argument in that paper first obtains an $L_p$ estimate with a factor involving $(p-1)^{-\frac{5}{2}}$ and then approaches $L_1$ by choosing $p$ close to one, whereas here the probability representation permits descent directly on the $L_1$ ball.

For bounded $L$, it remains natural to ask whether the proportional bound can be reduced further to $C\sqrt{\log n}$. The Walsh-system obstruction discussed in the appendix of \cite{GMPT2008} shows why a logarithmic loss cannot be removed altogether in general. Our proof leaves the additional factor $\sqrt{\log\log n}$, and removing it would require either a sharper way of representing the relevant functions or a probability argument which uses more than the uniform entropy bound of Lemma~\ref{lem:comparator}.

Another question is whether the uniform empirical estimate can be accompanied by sharper probability bounds for a model of sampling with fixed cardinality. The present proof uses an estimate in expectation together with Markov's inequality, which suffice for existence and for the simultaneous comparison on two complementary spans, but the polynomial failure bounds of the first paper in the series do not transfer automatically to this argument, and such a transfer would require a separate concentration estimate.

Finally, the predictors used here are devices for proving a bound on descriptions, and their updates may involve integration over an arbitrary probability space, so turning this proof into an efficient method for constructing the selected subsystem is a different problem. What the argument supplies, without any algorithmic assumption, is a complete passage from a decreasing entropy potential to a uniform norm inequality.


\begin{thebibliography}{99}

\bibitem{BeckTeboulle}
A. Beck and M. Teboulle,
Mirror descent and nonlinear projected subgradient methods for convex optimization,
Operations Research Letters 31 (2003), no. 3, 167--175.
\href{https://doi.org/10.1016/S0167-6377(02)00231-6}{doi:10.1016/S0167-6377(02)00231-6}.

\bibitem{Bourgain1989}
J. Bourgain,
Bounded orthogonal systems and the $\Lambda(p)$-set problem,
Acta Mathematica 162 (1989), 227--245.
\href{https://doi.org/10.1007/BF02392838}{doi:10.1007/BF02392838}.

\bibitem{Bourgain2014}
J. Bourgain,
An improved estimate in the restricted isometry problem,
in Geometric Aspects of Functional Analysis,
Lecture Notes in Mathematics 2116, Springer, Cham, 2014, 65--70.
\href{https://doi.org/10.1007/978-3-319-09477-9_5}{doi:10.1007/978-3-319-09477-9\_5}.

\bibitem{BIK1}
W. Burstein, A. Iosevich, and B. Krause,
Bourgain's $\Lambda(p)$ selection theorem: a greedy proof with polynomial failure bounds,
working manuscript, 2026. First paper in the series.

\bibitem{BIK2}
W. Burstein, A. Iosevich, and B. Krause,
Restricted isometry of sampled Fourier and Hadamard matrices via entropic descent,
working manuscript, 2026. Second paper in the series.

\bibitem{GMPT2007}
O. Gu\'edon, S. Mendelson, A. Pajor, and N. Tomczak-Jaegermann,
Subspaces and orthogonal decompositions generated by bounded orthogonal systems,
Positivity 11 (2007), no. 2, 269--283.
\href{https://maths-people.anu.edu.au/~mendelso/papers/GMPT-Hadamard-31a-08.pdf}{Author manuscript}.

\bibitem{GMPT2008}
O. Gu\'edon, S. Mendelson, A. Pajor, and N. Tomczak-Jaegermann,
Majorizing measures and proportional subsets of bounded orthonormal systems,
Revista Matem\'atica Iberoamericana 24 (2008), no. 3, 1075--1095.
\href{https://doi.org/10.4171/RMI/567}{doi:10.4171/RMI/567}.

\bibitem{IKLM2024}
A. M. Iosevich, B. S. Kashin, I. V. Limonova, and A. Mayeli,
Subsystems of orthogonal systems and the recovery of sparse signals in the presence of random losses,
Russian Mathematical Surveys 79 (2024), no. 6, 1095--1097.
\href{https://doi.org/10.4213/rm10213e}{doi:10.4213/rm10213e}.

\bibitem{IM2025}
A. Iosevich and A. Mayeli,
Uncertainty principles, restriction, Bourgain's $\Lambda_q$ theorem, and signal recovery,
Applied and Computational Harmonic Analysis 76 (2025), article 101734.
\href{https://doi.org/10.1016/j.acha.2024.101734}{doi:10.1016/j.acha.2024.101734}.

\end{thebibliography}
\end{document}